\documentclass{amsart}
\usepackage{graphicx,amssymb,amsmath, mathtools,amsxtra,mathrsfs,dsfont}

\usepackage{tikz-cd}

\usepackage[new]{old-arrows}

\usepackage[square,numbers]{natbib}

\usepackage{subfigure}

\usepackage{fancybox}

\usepackage[colorlinks = true, citecolor = blue, urlcolor = blue]{hyperref}

\usepackage[all]{xy}

\usepackage{cleveref}

\usepackage{float}

\usepackage{multirow}

\newtheorem{theorem}{Theorem}[section]
\newtheorem{lemma}[theorem]{Lemma}
\newtheorem{prop}[theorem]{Proposition}

\theoremstyle{definition}

\newtheorem{remark}[theorem]{Remark}

\def\h(#1,#2){\mbox{Hom}\left(#1,#2\right)}

\def\t(#1,#2){\mbox{Tor}\left(#1,#2\right)}
\def\e(#1,#2){\mbox{Ext}\left(#1,#2\right)}

\def\CP{\mathbb{C}\mathbf{P}}
\def\RP{\mathbb{R}\mathbf{P}}

\def\C{\mathbb{C}}
\def\F{\mathbb{F}}

\def\N{\mathbb{N}}

\def\R{\mathbb{R}}

\def\Z{\mathbb{Z}}

\def\xx{\mathbf{x}}
\def\yy{\mathbf{y}}

\def\per{\mathsf{per}}
\def\LPCA{\mathsf{LPCA}}
\def\LastLensComp{\mathsf{LastLensComp}}
\def\var{\mathsf{var}}
\def\cov{\mathsf{Cov}}
\def\linspan{\mathsf{span}}

\def\dgm{\mathsf{dgm}}

\def\<{\langle}
\def\>{\rangle}

\newcommand{\ppart}[1]{\left|#1 \right|_{+}}
\newcommand{\norm}[1]{\left|#1 \right|}

\DeclareMathOperator*{\argmin}{\arg\!\min\;}
\DeclareMathOperator*{\argmax}{\arg\!\max\;}

\title{Coordinatizing Data With Lens Spaces and Persistent Cohomology}

\author{
Luis Polanco }
\author{
Jose A. Perea 
}

\index{Polanco, Luis}

\address{Department of Computational Mathematics, Science \& Engineering, Department of Mathematics,
         Michigan State University, East Lansing,
         MI, USA.}
\email{polanco2@msu.edu}

\index{Parea, Jose}

\address{Department of Computational Mathematics, Science \& Engineering, Department of Mathematics,
         Michigan State University, East Lansing,
         MI, USA.}
\email{joperea@math.msu.edu}



\begin{document}
\bibliographystyle{abbrvnat}
\maketitle

\begin{abstract}
We introduce here a framework to construct coordinates in \emph{finite} Lens spaces for data with nontrivial  1-dimensional $\Z_q$ persistent cohomology,
for  $q>2$ prime.
Said coordinates  are defined on an open neighborhood of the data, yet
constructed with only a small subset of landmarks. We also introduce a dimensionality reduction scheme in $S^{2n-1}/\Z_q$ (Lens-PCA: $\mathsf{LPCA}$),
and demonstrate the efficacy of the pipeline
$ \Z_q$-persistent cohomology
$\Rightarrow$ $S^{2n-1}/\Z_q$ coordinates
$\Rightarrow$  $\mathsf{LPCA}$,
for  nonlinear (topological) dimensionality reduction.
\end{abstract}

\section{Introduction}
One of the main questions in   Topological Data Analysis (TDA) is how to use topological signatures like persistent (co)homology \cite{perea2018brief} to infer spaces  parametrizing a given
data set
\cite{carlsson_2014, klein_dictionary,Carlsson_natural_images}.
This is relevant in nonlinear dimensionality reduction since the presence
of nontrivial  topology---e.g., loops, voids, non-orientability, torsion, etc---can prevent accurate descriptions with low-dimensional Euclidean coordinates.

Here we seek to address this problem  motivated by two facts.
The first:
If $G$ is a topological abelian group,
then one can associate to  it a contractible space, $EG$, equipped with a free right $G$-action.
For instance, if $G = \Z$, then
$\R$ is a model for $E\Z$, with right $\Z$-action
$\R\times \Z \ni(r,n)\mapsto r + n\in \R$.
The quotient $BG := EG/G$ is
called the classifying space of $G$ \cite{milnor1956construction}. In particular
$B\Z \simeq S^1$, $B\Z_2 \simeq \RP^\infty$, $BS^1 \simeq \CP^\infty$ and $B\Z_q \simeq S^\infty/\Z_q$;
 here $\simeq$ denotes homotopy equivalence.
The second fact:
If $B$ is a topological space and
$\mathscr{C}_G$ is the sheaf  over $B$ (defined in \cite{miranda1995algebraic})
sending  each  $U\subset B$ open to the abelian group of
continuous maps from $U$ to $G$,
then $\check{H}^1(B;\mathscr{C}_G)$---the first \v{C}ech cohomology group of $B$
with coefficients in $\mathscr{C}_G$---is in bijective correspondence with
$[B\, ,\, BG]$---the set of homotopy classes of continuous maps from $B$ to the classifying space $BG$.
%
%
%
This bijection is a manifestation
of the Brown representability theorem
\cite{brown1962cohomology}, and implies,
in so many words, that \v{C}ech cohomology classes can be
represented as coordinates with values in a classifying space (like $S^1$ or $S^\infty/\Z_q$).

For  point cloud data---i.e., for a finite subset $X$ of an ambient metric space $(M,d)$---one does not compute \v{C}ech cohomology, but rather
\emph{persistent cohomology}.
Specifically,
the persistent cohomology of the Rips
filtration on the data set $X$ (or
a subset of landmarks $L$).
The first main result of this paper
contends  that  steps one through three   below
mimic the bijection
$\check{H}^1(B ;\mathscr{C}_{\Z_q}) \cong[B , S^\infty/\Z_q]$
for $B\subset M$ an open neighborhood of $X$:
\begin{enumerate}
    \item Let $(M,d)$ be a  metric space and  let $L\subset  X\subset M$  be finite. $X$ is the data
        and $L$ is a set of landmarks.

    \item For a prime $q>2$ compute $PH^1(\mathcal{R}(L);\Z_q)$;
        the 1-dim $\Z_q$-persistent cohomology of the Rips filtration on $L$. If the corresponding
        persistence diagram
        $\dgm(L)$ has an element $(a,b)$ so that $2a < b$, then let $a \leq \epsilon < b/2$ and choose a  representative cocycle $\eta \in Z^1(R_{2\epsilon}(L); \Z_q)$ whose cohomology class has  $(a,b)$ as birth-death pair.

    \item Let $B_\epsilon(l)$ be the open ball in $M$ of radius $\epsilon$ centered at
    $l\in L = \{l_1,\ldots, l_n\}$, and let $\varphi = \{\varphi_l\}_{l\in L}$
    be a partition of unity subordinated to $\mathcal{B} = \{B_\epsilon(l)\}_{l\in L}$.
    If $\zeta_q \neq 1$ is a $q$-th root of unity,
    then the cocycle $\eta$
    yields a  map
    $f: \bigcup\mathcal{B}\longrightarrow L_q^n$
    to the Lens space $L^n_q = S^{2n-1}/\Z_q$,
    given in homogeneous coordinates   by the formula
    \[
    B_\epsilon(\ell_j) \ni b \;,\;
    f(b) = \left[\sqrt{\varphi_1(b)} \zeta_q^{\eta_{j1}} :\cdots: \sqrt{\varphi_n(b)} \zeta_q^{\eta_{jn}}\right]
    \]
    where $\eta_{jk}\in \Z_q$  is
        the value of the cocycle $\eta$ on the edge
        $\{l_j,l_k\} \in R_{2\epsilon}(L)$.
\end{enumerate}

If  $X\subset \bigcup\mathcal{B}$,
then $f(X) = Y \subset L^n_q$ is the representation of the data---in a potentially high dimensional Lens space---corresponding to the cocycle $\eta$.
The second contribution of this paper
is a dimensionality reduction procedure
in $L^n_q$ akin to Principal Component
Analysis,   called $\mathsf{LPCA}$.
This allows us to produce from $Y$,
a family of point clouds
$P_k(Y) \subset L^k_q$, $1\leq k \leq n$, $P_n(Y) = Y$,
minimizing an appropriate notion of distortion.
These are the Lens coordinates of
 $X$ induced by the cocycle $\eta$.

This work, combined with
\cite{circular_coord, projective_coord},
should be seen as one of the final steps in completing
the program of using the classifying space $BG$, for $G$ abelian and finitely generated, to produce coordinates for data with nontrivial underlying $1^{st}$ cohomology.
Indeed, this follows from the fact
that $B(G\oplus G') \simeq BG \times BG'$,
and that if $G$ is finitely generated
and abelian, then it is ismorphic to
$\Z^n \oplus \Z_{n_1}\oplus \cdots \oplus \Z_{n_r}$
for unique integers $n,n_1,\ldots, n_r \geq 0 $.

\section{Preliminaries}
\subsection{Persistent Cohomology}
A family
$
\mathcal{K} = \{K_{\alpha}\}_{\alpha\in \R}$  of simplicial complexes
 is called a filtration if $K_\alpha \subset K_{\alpha'}$ whenever  $\alpha \leq \alpha'$.
If $\F$ is a field and $i\geq 0$ is an integer, then the direct sum
$PH^i(\mathcal{K}; \F) : = \bigoplus\limits_{\alpha} H^i(K_\alpha ; \F)$
of cohomology groups is called the $i$-th dimensional \textbf{$\F$-persistent cohomology} of $\mathcal{K}$.
A theorem of Crawley-Boevey \cite{crawley2015decomposition}
contends that if  $H^i(K_\alpha;\F)$ is  finite dimensional   for
each $\alpha$, then the isomorphism type of $PH^i(\mathcal{K};\F)$---as a persistence module---is uniquely
determined by a multiset (i.e., a set whose elements may appear with repetitions)
\[
\dgm \subset \{(\alpha, \alpha') \in [-\infty, \infty]^2 : \alpha \leq \alpha' \}\]
called the \textbf{persistence diagram} of $PH^i(\mathcal{K};\F)$.
Pairs $(\alpha,\alpha')$ with large persistence $\alpha' - \alpha$, are
indicative of stable topological features throughout the filtration $\mathcal{K}$.

Persistent cohomology is used in TDA to quantify the topology underlying
a data set.
There are two widely used filtrations associated to a subset $X$ of a
metric space $(M,d)$, the \textbf{Rips filtration}  $\mathcal{R}(X) = \{R_\alpha(X)\}_{\alpha}$ and
the \textbf{\v{C}ech filtration}
 $\check{\mathcal{C}}(X) = \{\check{C}_\alpha(X)\}_\alpha$.
Specifically,   $R_\alpha(X)$ is the set of nonempty
finite subsets of $X$ with diameter less than $\alpha$, and
$\check{C}_\alpha(X)$ is the nerve of the collection $\mathcal{B}_\alpha$
of open balls  $B_\alpha(x) \subset M$ of radius $\alpha$, centered at $x\in X$.
In other words, $\check{C}_\alpha(X) = \mathcal{N}(\mathcal{B}_\alpha)$.
Generally  $\mathcal{R}(X)$ is more easily computable,
but $\check{\mathcal{C}}(X)$ has better theoretical properties (e.g., the Nerve theorem \cite[4G.3]{hatcher2002algebraic}).
Their relative weaknesses are ameliorated by noticing that
\[R_\alpha(X) \subset \mathcal{N}(\mathcal{B}_\alpha) \subset R_{2\alpha}(X)\]
for all $\alpha$, and using both filtrations in analyses: Rips for computations, and \v{C}ech for theoretical inference.

\subsection{Lens Spaces}

Let $q \in \N$ and let $\zeta_q \in \C$ be a primary $q$-th root of unity. Fix $n\in \N$ and let $q_1,\ldots, q_n \in \N$ be relatively prime to $q$. We define the \textbf{Lens space} $L^n_q(q_1, \ldots, q_n)$ as the quotient of
$S^{2n -1} \subset \C^n$ by the $\Z_q$ right action
\[
 [z_1,\ldots, z_n] \cdot g  : = \left[  z_1\zeta_q^{ q_1g} ,\ldots, z_n\zeta_q^{ q_ng} \right]
\] with simplified notation $L^n_q := L^n_q(1, \ldots, 1 )$. Notice that   when $q=2$
and $q_1 = \cdots = q_n = 1$, then the right action described above is the antipodal map of $S^{2n-1}$, and therefore $L_2^n = \RP^{2n-1}$.
Similarly, the infinite Lens
space $L^\infty_q = L^\infty_q(1,1,\ldots )$ is defined
as the quotient of the infinite unit sphere $S^\infty \subset \C^\infty $,
by the action of $\Z_q$
induced by scalar-vector multiplication
by powers of   $\zeta_q$.
%
%
%

\subsubsection{A Fundamental domain for $L^2_q(1, p)$}\label{subsec:fundamental_domain}
In what follows we describe a convenient model for both  $L^2_q(1, p)$ and a fundamental domain thereof.
This model will allow us to provide visualizations in Lens spaces towards the end of the paper.
Let $D^3$ be the set of points  $\xx \in \R^3$ with $\|\xx\|\leq 1$, and let
$D_+$ ($D_{-}$) be the upper (lower) hemisphere of $\partial D^3$, including the equator.
Let $r_{p/q} : D_{+} \longrightarrow D_{+}$
 be counterclockwise rotation by $2\pi p/q$ radians  around the $z$-axis,
 and let $\rho : D_+ \longrightarrow D_{-}$
 be the reflection $\rho(x,y,z) = (x,y,-z)$.
 Then,
 $L^2_q(1,p) $ is homeomorphic to $D^3/\sim$,
 where $\xx \sim \yy$ if and only if
 $\xx \in D_{+}$  and $\yy =  \rho\circ r_{p/q}(\xx)$.

\subsection{Principal Bundles}
Let $B$ be a  topological  space with base point $b_0 \in B$.
One of the most transparent methods for producing an explicit bijection
between
$\check{H}^1(B;\mathscr{C}_{\Z_q}) $ and $ [B,L^\infty_q]$ is via the
theory of Principal bundles. We present a terse introduction here,
but direct the interested reader to \cite{husemoller1994fibre} for  details.
A continuous map $ \pi :P \longrightarrow B$ is said to be a \textbf{fiber bundle} with fiber $F = \pi^{-1}(b_0)$ and total space $P$, if $\pi$ is surjective, and
every $b\in B$ has an open neighborhood $U \subset B$ as well as a homeomorphism $\rho_U : U\times F \longrightarrow \pi^{-1}(U)$,
    so that $\pi \circ \rho_U(x,e) = x$ for every $(x,e) \in U \times F$.

Let $(G,+)$ be an abelian topological group.
A fiber bundle $\pi : P \longrightarrow B$ is said to be a \textbf{principal $G$-bundle} over $B$,
if $P$ comes equipped with a  free right $G$-action $ P\times G  \ni (e,g)\mapsto e\cdot g\in P$
which is transitive in $\pi^{-1}(b)$ for every $b\in B$.
Moreover, two principal
$G$-bundles
$\pi: P \longrightarrow B$
and $\pi' : P' \longrightarrow B$
are isomorphic, if there exits
a  homeomorphism
$\Phi: P \longrightarrow P'$,
with $\pi' \circ \Phi = \pi$
    and so that $\Phi(e\cdot g) = \Phi(e)\cdot g$ for all
$(e,g)\in P\times G$.
Given an open cover $\mathcal{U}= \{U_j\}_{j\in J}$ of $B$,
a \textbf{\v{C}ech cocycle}
\[
 \eta = \{\eta_{jk}\}  \in \check{Z}^1(\mathcal{U}; \mathscr{C}_G )
\]
is a collection of continuous maps $\eta_{jk} : U_j \cap U_k \longrightarrow G$
so that $\eta_{jk}(b) + \eta_{kl}(b) = \eta_{jl}(b)$ for every $b\in U_j \cap U_k \cap U_l$.
Given such a cocycle,
one can construct
 a principal $G$-bundle over $B$ with total space
\[
P_\eta = \left(\bigcup_{j\in J} U_j \times \{j\} \times G \right) / \sim
\]
where $(b,j,g) \sim (b, k , g  + \eta_{jk}(b))$ for every $b\in U_j \cap U_k$,
and  $\pi: P_\eta \longrightarrow B$ sends the class of $(b,j,g)$ to $b\in B$.

\begin{theorem}\label{bijection}
If $\mathsf{Prin}_G(B)$ denotes
the set of isomorphism classes of
principal $G$-bundles over $B$, then
\[
\begin{array}{ccl}
\check{H}^1(B;\mathscr{C}_G) & \longrightarrow & \mathsf{Prin}_G(B) \\
\left[\eta\right] &\mapsto & [P_\eta]
\end{array}
\]
is a bijection.
\end{theorem}
\begin{proof}
See 2.4 and 2.5 in \cite{circular_coord}
\end{proof}

Now, let us see describe the relation between principal
$G$-bundles over $B$, and maps from $B$ to the classifying space $BG$.
Indeed,
let $
\jmath : EG \longrightarrow BG = EG/G$
be the quotient map.
Given   $h: B \longrightarrow BG$ continuous, the pullback
$h^*EG$ is   the principal $G$-bundle over $B$
with total space $\{(b,e) \in B\times EG : h(b) = \jmath(e)\}$,
and  projection map  $(b,e)\mapsto b$.
Moreover,
\begin{theorem}\label{thm:IsoHomoPrinG}
Let $[B,BG]$ denote the set of homotopy class of maps from $B$ to the classifying space $BG$.
Then, the function
\[
\begin{array}{ccl}
  [B, BG] & \longrightarrow &  \mathsf{Prin}_G(B)\\
 \; [ h] & \mapsto & [h^*EG]
\end{array}
\]
is a  bijection.
\end{theorem}
\begin{proof}
See \cite{husemoller1994fibre}, Chapter 4: Theorems 12.2 and 12.4.
\end{proof}

In summary,
given a principal $G$-bundle $\pi : P \longrightarrow B$, or its corresponding \v{C}ech cocycle $\eta$,
there exists a continuous map $h: B \longrightarrow BG $ so
that $h^*EG $ is isomorphic to $(\pi,P)$, and  the choice of $h$ is unique up to homotopy.
Any such choice is called a classifying map for   $\pi: P \longrightarrow B$.

\section{Main Theorem: Explicit Classifying Maps for $L^\infty_q$} \label{sec:classifying_map}
The goal of this section is to show how one can go from a singular cocycle
$\eta \in Z^1(\mathcal{N}(\mathcal{U}); \Z_q)$ to an explicit map
$f: \bigcup \mathcal{U} \longrightarrow L^\infty_q$. All proofs are included in the Apendix.
Let
$J = \{1,\ldots, n\}$,
let $\mathcal{U} = \{U_j\}_{j\in J}$ be  an
open cover for   $B$,
and let
$\{\varphi_j\}_{j\in J}$ be a partition of unity dominated by $\mathcal{U}$.
If  $\eta = Z^1(\mathcal{N}(\mathcal{U});\Z_q)$  and $\zeta_q$ is a primitive $q$-th root of unity,
  let
$f_j : U_j \times \{j\} \times \Z_q \longrightarrow  S^{2n-1} \subset \C^n$
be
\[
f_j (b,j,g) = \left[\sqrt{\varphi_1(b)}\zeta_q^{ (g +  \eta_{j1})}, \ldots, \sqrt{\varphi_n(b)}\zeta_q^{( g + \eta_{jn})}\right]
\]

If $b\in U_j \cap U_k$, then $f_j (b, j, g) = f_k (b , k, g + \eta_{jk})$
and  we get an induced map $\Phi: P_\eta \longrightarrow S^{2n-1}\subset S^\infty$
 taking the class
of $(b,j,g)$ in the quotient $P_\eta$ to $f_j(b,j,g)$.

\begin{prop}\label{prop:F_equivariants}
$\Phi$ is well defined and $\Z_q$-equivariant.
\end{prop}
\begin{proof}
Take $[b,j,g] \in P_\eta$ and consider a different representative of the class. Namely, an element $(b,k,g+\eta_{jk})$ such that $b\in U_j\cap U_k$. By definition of $\Phi$, we have $\Phi([b,j,g]) = f_j(b,j,g)$ and $\Phi([b,k,g+\eta_{jk}]) = f_k(b,k,g+\eta_{jk})$. And since $f_j (b, j, g) = f_k (b , k, g + \eta_{jk})$, we have that \[ \Phi([b,j,g]) = \Phi([b,k,g+\eta_{jk}]), \] which shows that $\Phi$ is well defined.

To see that $\Phi$ is $\Z_q$-equivariant, take $m\in\Z_q$ for any $m=0,\dots,q-1$ and compute
\begin{align*}
    \Phi & ([b,j,g]) \cdot m \\
    &= \left[\sqrt{\varphi_1(b)}\zeta_q^{ (g + m +  \eta_{j1})}, \ldots, \sqrt{\varphi_n(b)}\zeta_q^{( g + m + \eta_{jn})}\right] \\
    &= f_j(b, j, g + m) = \Phi([b, j, g + m]) \\
    &= \Phi([b, j, g]\cdot m).
\end{align*}
\end{proof}

Let $p:S^{2n-1}\longrightarrow L_q^n$ be the quiotient map.
Since $\Phi:P_\eta\longrightarrow S^{2n-1}\subset S^\infty$ is $\Z_q$-equivariant, it induces a map $f: B \longrightarrow L_q^n \subset L_q^\infty$ such that $p\circ \Phi = f \circ \pi$.
By construction of $\pi: P_\eta\longrightarrow B$, $f(\pi([b,j,g])) = f(b)$ for any $g\in \Z_q$. In particular for $0\in\Z_q$ \begin{equation}\label{eqn:f}
  U_j \ni b \;\; , \;\;   f(b) = \left[\sqrt{\varphi_1(b)}\zeta_q^{\eta_{j1}} : \cdots : \sqrt{\varphi_n(b)}\zeta_q^{\eta_{jn}}\right]
\end{equation}

\begin{remark}
The notation $ \left[a_1 : \cdots : a_n\right] $ corresponds to  homogeneous coordinates in $S^{2n-1}/\Z_q$.
In other words,  $\left[a_1 : \cdots : a_n\right] = \{[a_1\cdot \alpha,\dots,  a_n\cdot \alpha] \in S^{2n-1} : \alpha\in\Z_q \}$.
\end{remark}

\begin{theorem}\label{thm:ClassifyingMap}
The map $f$ classifies the $\Z_q$-principal  bundle $P_\eta$ associated to the cocycle
$\eta \in Z^1(\mathcal{N}(\mathcal{U});\Z_q)$.
\end{theorem}
\begin{proof}
First we need to see that $f$ is well defined. Let $b\in U_j\cap U_k$, therefore \begin{align*}
    p(\Phi([b,j,0])) &= \left[\sqrt{\varphi_1(b)}\zeta_q^{\eta_{j1}} : \cdots : \sqrt{\varphi_n(b)}\zeta_q^{\eta_{jn}}\right] \\
    &= p(\Phi([b,k,0)).
\end{align*} This shows that $f(b)$ is independent of the open set containing $b$.

Hence $(\Phi,f):(P_\eta,\pi,B) \rightarrow (S^{2n-1},\pi, L^n_q)$ is a morphism of principal $\Z_q$-bundles, and  by   [\cite{husemoller1994fibre}, Chapter 4: Theorem 4.2] we conclude that $P_\eta$ and $f^*(S^{2n-1})$ are isomorphic principal $\Z_q$-bundles over $B$.
\end{proof}

\section{Lens coordinates for data}

Let $(M,d)$ be a metric space and let $L\subset M$ be a finite subset. We will use the following notation from now on:
$B_\epsilon(l) = \{y\in M : d(y,l)<\epsilon\}$,
$\mathcal{B}_\epsilon  = \{B_\epsilon(l)\}_{l\in L}$,
and
$ L^\epsilon = \bigcup \mathcal{B}_\epsilon$.
Given a data set
$X\subset M$,
our goal will be to
choose $L\subset X$, a suitable $\epsilon$ such that $X\subset L^\epsilon$,
and a cocycle $\eta \in Z^1(\mathcal{N}(\mathcal{B}_\epsilon);\Z_q)$.
\Cref{eqn:f}
yields a map $f:L^\epsilon \rightarrow L^\infty_q$ defined for every point in $X$, but constructed from a much smaller subset of landmarks.
Next we describe the details of this construction.

\subsection{Landmark selection}\label{subsec:landmark}
We select the landmark set $L\subset X$
either at random or through \verb"maxmin " sampling. The latter   proceeds inductively as follows:
Fix $n \leq |X|$, and let $l_1\in X$ be chosen at random. Given $l_1, \dots, l_j \in X$ for $j<n$, we  let  $ l_{j+1} = \argmax\limits_{x\in X} \min \{d(x,l_1), \dots, d(x,l_j)\}$.

\subsection{A Partition of Unity subordinated to $\mathcal{B}_\epsilon$}

Defining $f$ requires a partition of unity subordinated to $\mathcal{B}_\epsilon$. Since $\mathcal{B}_\epsilon$ is an open cover composed of metric balls, then we can provide an explicit construction. Indeed, for $r\in \R$ let $\ppart{r} := \max\{r,0\}$, then

\begin{equation}\label{eqn:partition_unity}
    \varphi_l(x) := \ppart{\epsilon - d(x,l)} \Big/ \sum_{l'\in L} \ppart{\epsilon - d(x,l')}
\end{equation}
is a partition of unity subordinated to
$\mathcal{B}_\epsilon$.
\subsection{From Rips to \v{C}ech to Rips}\label{subsec:cech-rips}
As we alluded to in the introduction,
a persistent cohomology calculation is an appropriate vehicle to select a scale
$\epsilon$ and a candidate cocycle $\eta$.
That said, determining $\eta\in Z^1(\mathcal{N}(\mathcal{B}_\epsilon), \Z_q)$ would require computing $\mathcal{N}(\mathcal{B}_\epsilon)$ for all $\epsilon$, which in general is an expensive procedure.
Instead we will use the homomorphisms
\[\xymatrix{
    H^1(\mathcal{R}_{2\epsilon}(L)) \ar[r]^{i^*} \ar@/_1pc/[rr]_{\iota} & H^1(\mathcal{N}(\mathcal{B}_\epsilon)) \ar[r]  & H^1(\mathcal{R}_{\epsilon}(L))
}\]
induced by the appropriate inclusions.
Indeed, let $\tilde{\eta} \in Z^1(\mathcal{R}_{2\epsilon}(L); \Z_q)$ be such that $[\tilde{\eta}]\not\in \ker(\iota)$. 
This is where we use the persistent cohomology of $\mathcal{R}(L)$.
Since the previous diagram commutes, then $[\tilde{\eta}]\not\in \ker(i^*)$, so $i^*([\tilde{\eta}])\neq 0$ in $H^1(\mathcal{N}(\mathcal{B}_\epsilon); \Z_q)$. We will let  $[\eta] = i^*([\tilde{\eta}])$  be the class that we use in \Cref{thm:ClassifyingMap}. 
However,

\begin{prop}\label{prop:rips-cech}
If $b\in \mathcal{B}_\epsilon(l_j)$ and   $1 \leq k \leq  n$, then \[
\sqrt{\varphi_k(b)}\zeta_q^{\eta_{jk}} = \sqrt{\varphi_k(b)}\zeta_q^{\tilde{\eta}_{jk}}.
\] That is, we can compute Lens coordinates using only the Rips filtration on the landmark set.
\end{prop}
\begin{proof}
First of all, $\mathcal{R}_{2\epsilon}(L)^{(0)} = \mathcal{N}(\mathcal{B}_\epsilon)^{(0)} = L$. If $b\not\in B_\epsilon(l_k)$, then $\varphi_k(b) = 0$ and therefore the equality holds. If on the other hand $b\in B_\epsilon(l_k) \cap B_\epsilon(l_j)$, then $ \{ j,k \} \in \mathcal{N}(\mathcal{B}_\epsilon)^{(1)} \subset \mathcal{R}_{2\epsilon}(L)^{(1)}$. In which case, by definition of $i^*$, we have $\tilde{\eta}_{jk} = \eta_{jk}$.
\end{proof}

\section{Dimensionality Reduction in $L^n_q$ via Principal Lens Components}

\Cref{eqn:f} gives an explicit formula for the classifying map $f:B\longrightarrow L_q^n$. By construction, the dimension of $L_q^n$ depends on the number $n$ of landmarks selected, which in general can be  large.
The main goal of this section is  to construct a dimensionality reduction procedure in $L_q^n$ to address this shortcoming.
To this end, we define the distance $d_L :  L_q^n \times  L_q^n\longrightarrow [0,\infty)$ as
\begin{equation}\label{eqn:hausdorff_distance}
    d_L([x],[y]) := d_H(x\cdot\Z_q\,,\, y\cdot\Z_q)
\end{equation} where $d_H$ id the Hausdorff distance for subsets of $S^{2n-1}$.


\begin{prop}\label{prop:hausdorff_distance} Let $[x], [y] \in L_q^n$, then
\[d_L([x],[y]) = d(x, y\cdot \Z_q) = \min_{g\in\Z_q} d(x, y \cdot g).\]
\end{prop}

\begin{proof}
For $x,y\in \C^n$ let $\left\< x,y \right\>_\R := \mathsf{real}( \<x,y\>_\C )$.
By definition of Hausdorff distance, we have that \begin{align*}
d_L([x],[y]) &= \max\left\{ \max_{g\in\Z_q}\min_{h\in\Z_q}\arccos(\< x\cdot g, y\cdot h \>_\R) \right. , \\ & \left. \quad \max_{h\in\Z_q}\min_{g\in\Z_q}\arccos(\< x\cdot g, y\cdot h \>_\R) \right\}.
\end{align*}

Notice that
\begin{align*}
    \< x\cdot g, y\cdot h \>_\R &= \mathsf{real}\left( \left\< \zeta_q^g x, \zeta_q^h y \right\>_\C \right) \\
    &= \mathsf{real}\left( \left\< x, \zeta_q^{(h-g)} y \right\>_\C \right) \\
    &= \< x, y\cdot (h-g)\>_\R 
\end{align*}

And since $\Z_q$ is Abelian, then \begin{align*}
    \max_{h\in\Z_q}\min_{g\in\Z_q} & \arccos(\< x\cdot g, y\cdot h \>_\R) \\
    &= \max_{h\in\Z_q}\min_{g\in\Z_q}\arccos(\< x\cdot (g-h), y \>_\R) \\
    &= \max_{h\in\Z_q}\min_{g\in\Z_q}\arccos(\< x\cdot (-h), y\cdot (-g) \>_\R) \\
    &= \max_{h'\in\Z_q}\min_{g'\in\Z_q}\arccos(\< x\cdot h', y\cdot g' \>_\R).
\end{align*}

Thus \[
d_L([x],[y]) = \max\limits_{g\in\Z_q}\min_{h\in\Z_q}\arccos(\< x\cdot g, y\cdot h \>_\R). \] Furthermore $d_L([x],[y]) = \max\limits_{g\in\Z_q} d(x\cdot g, y\cdot \Z_q) = \max\limits_{g\in\Z_q} d(x, y\cdot (-g)\Z_q )$. 
Since $y\cdot \big((-g)\Z_q\big)  = y\cdot \Z_q $ for any $g\in\Z_q$, we obtain $d_L([x],[y]) = \max\limits_{g\in\Z_q} d(x, y\cdot \Z_q) = d(x, y\cdot \Z_q) = \min\limits_{h\in\Z_q} d(x, y\cdot h)$. 
\end{proof}

We will now describe a notion of \textbf{projection in $L_q^n$}
onto lower-dimensional Lens spaces.
Indeed,
let $u\in S^{2n-1} $. Since $\zeta_q^k w \in \linspan_\C(u)^\bot$ for any $k \in \Z_q$ and $w\in \linspan_\C(u)^\bot$, then \[ L_q^{n-1}(u) := ( \linspan_\C(u)^\bot \cap S^{2n-1} ) / \Z_q \] is isometric  to $L^{n-1}_q$.
Let $P_u^\bot(v)= v - \<v,u\>_\C u $
for $v\in \C^n$,
and if
$v \notin \mathsf{span}_\C(u)$, then we let
\[
\mathcal{P}_u([v]) := \left[  P_u^\bot (v)
\big/\| P_u^\bot (v) \| \right]
 \in L^{n-1}_q(u)\]
It readily follows that
$\mathcal{P}_u$ is well defined, and that
\begin{lemma}\label{lemma:ditance_to_proj}
For $ u \in S^{2n-1}$
and $v\notin \mathsf{span}_\C(u)$,
we have
\[
d_L([v], \mathcal{P}_u([v])) = d\left(v \;,\; P_u^\bot (v)\big/\| P_u^\bot (v) \| \right)
\]
where $d$ is the  distance on $S^{2n-1}$.
Furthermore, $\mathcal{P}_u ([v])$ is the point in $L_q^{n-1}(u)$ closest to $[v]$ with respect to $d_L$.
\end{lemma}
\begin{proof}
From \Cref{prop:hausdorff_distance} we know that \begin{align*}
    d_L([v], P_u^\bot([v])) &= \min_{g\in\Z_q} d(v, P_u^\bot([v])\cdot g ) \\
    &= \min_{g\in\Z_q} d\left(v, \frac{P_u^\bot (v)}{\| P_u^\bot (v) \|}\cdot g \right).
\end{align*} Let $g^* := \argmin\limits_{g\in\Z_q} d\left(v, \frac{P_u^\bot (v)}{\| P_u^\bot (v) \|}\cdot g \right)$, so we have \[ d_L([v], P_u^\bot([v])) = \arccos \left( \left< v , \frac{P_u^\bot (v)}{\| P_u^\bot (v) \|}\cdot g^* \right>_\R\right). \]

Notice that the argument of the $\arccos$  can be simplified as follows \begin{align*}
    \left< v , \frac{P_u^\bot (v)}{\| P_u^\bot (v) \|}\cdot g^* \right>_\R &= \left< \<v,u\>_\C u + P_u^\bot(v) , \frac{P_u^\bot (v)}{\| P_u^\bot (v) \|}\cdot g^* \right>_\R \\
    &= \left< \<v,u\>_\C u , \frac{P_u^\bot (v)}{\| P_u^\bot (v) \|}\cdot g^* \right>_\R \\
    &\quad + \left< P_u^\bot(v) , \frac{P_u^\bot (v)}{\| P_u^\bot (v) \|}\cdot g^* \right>_\R.
\end{align*} since $u$ and $P_u^\bot(v)$ are orthogonal in $\C^n$ then they are also orthogonal in $\R^{2n}$, making the then the firs summand on the right hand side equal to zero. Additionally since $\arccos$ as a real valued function is monotonically decreasing we have \[g^* = \argmax_{g\in\Z_q}  \frac{1}{\| P_u^\bot (v) \|} \left< P_u^\bot(v) , P_u^\bot (v)\cdot g \right>_\R .\]

Using the fact that the action of $\Z_q$ is an isometry (and therefore an operator of norm one) as well as the Cauchy-Schwartz inequality we obtain \begin{align*}
    \frac{ \left< P_u^\bot(v) , P_u^\bot (v)\cdot g \right>_\R }{\| P_u^\bot (v) \|}  &\leq \left\vert \frac{1}{\| P_u^\bot (v) \|} \left< P_u^\bot(v) , P_u^\bot (v)\cdot g \right>_\R \right\vert \\
    &\leq \frac{1}{\| P_u^\bot (v) \|} \| P_u^\bot(v) \| \| P_u^\bot(v)\cdot g \| \\
    &= \| P_u^\bot(v)\cdot g \| = \| P_u^\bot(v) \|.
\end{align*} And the equality holds whenever $g=e \in \Z_q$, so we must have $g^* = e$.

Let $[w]\in L_q^{n-1}(u)$, so $w\in \linspan_\C^\bot(u)$ which implies that for any $h\in\Z_q$ \[\<u, w\cdot h \>_\C = \sum_k u_k(\overline{\zeta_q^h w_k}) = \zeta_q^{-h}\sum_k u_k\overline{w_k} = \zeta_q^{-h}\<u, w\> = 0. \] In other words $w\cdot h \in \linspan_\C^\bot(u)$ for any $h\in\Z_q$.

Thus by the Cauchy-Schwartz inequality \begin{align*}
    \< v,w\cdot h \>_\R &= \< \<v,u\>_\C u + P_u^\bot(v), w\cdot h \>_\R = \< P_u^\bot(v), w\cdot h \>_\R \\
    &\leq \vert \< P_u^\bot(v), w\cdot h \>_\R \vert \leq \|P_u^\bot(v)\| \|w\cdot h\| \\
    &= \|P_u^\bot(v)\| \|w\| = \|P_u^\bot(v)\|,
\end{align*} since the action of $\Z_q$ is an isometry and $w\in S^{2n-1}$.

Finally since $\arccos$ is decreasing \[d_L([v], P_u^\bot([v])) = \arccos(\|P_u^\bot(v)\|) \leq \arccos(\< v,w\cdot h \>_\R)\] for all $h\in\Z_q$, thus $d_L([v], P_u^\bot([v])) \leq d_L([v],[w])$.
\end{proof}

This last result suggests that a PCA-like approach is possible for  dimensionality reduction in Lens spaces.
Specifically,  for $Y = \{[y_1], \dots, [y_N]\}\subset L_q^n$, the goal is to find $u \in S^{2n-1}$ such that $L_q^{n-1}(u)$ is the best $(n-1)$-Lens space approximation to $Y$,
then project $Y$  onto $L^{n-1}_q(u)$ using $\mathcal{P}_u$,
and repeat the process iteratively reducing the dimension by 1 each time.
At each stage, the appropriate constrained optimization problem is
\begin{align*}
    u^* &= \argmin_{u\in \C^n , \|u\|=1} \sum_{j=1}^N d_L([y_j], \mathcal{P}_u([y_i]))^2 \\
    &= \argmin_{u\in \C^n , \|u\|=1} \sum_{j=1}^N \left( \frac{\pi}{2} - \arccos(\vert \<y_i, u\> \vert ) \right)^2
\end{align*}
which can be linearized using the Taylor series expansion of $\arccos(\theta)$ around $0$.
Indeed, $\vert \frac{\pi}{2} -\arccos(\theta) \vert \approx \vert\theta\vert$ to third order, and thus \[ u^* \approx \argmin_{u\in \C^n , \| u\|=1} \sum_{j=1}^N \vert \<y_i, u\> \vert^2 . \] This approximation is a linear least square problem whose solution is given by the eigenvector corresponding to the smallest eigenvalue of the covariance matrix \[ \cov\left( y_1, \dots, y_N \right) = \left[\begin{smallmatrix}
\vert & & \vert \\
y_1 & \cdots & y_N \\
\vert & & \vert
\end{smallmatrix}\right]
\left[ \begin{smallmatrix}
- & \overline{y}_1 & - \\
 & \vdots &  \\
- & \overline{y}_N & -
\end{smallmatrix} \right]. \]

Moreover, for any $\alpha_1, \dots, \alpha_N\in S^1 \subset \C$  we have that
$ \cov\left( \alpha_1 y_1, \dots, \alpha_N y_N \right) = \cov\left( y_1, \dots, y_N \right)$, so $\cov(Y)$ is well defined for $Y\subset L_q^n$.

\subsection{Inductive construction of LPCA}

Let  $v_n = \LastLensComp(Y)$ be the eigenvector of $\cov(Y)$ corresponding to the smallest eigenvalue.
Assume that we have constructed $v_{k+1}, \dots, v_n \in S^{2n-1}$ for $1 <k < n  $, and let $\{u_1, \dots, u_k \}$ be an orthonormal basis for $\linspan_\C(v_{k+1}, \dots, v_n )^\bot$.
Let $ U_k  \in \C^{n \times k}$
be the matrix with columns
$ u_1, \ldots, u_k $, and let $U_k^\dag $ be its conjugate transpose.
We define the \textbf{$k$-th Lens Principal component} of $Y$ as the vector
\[
v_k \;:= \; U_k \cdot \LastLensComp\left(\frac{U_k^\dag y_1}{\|U_k^\dag y_1\|} , \ldots, \frac{U_k^\dag y_N}{\|U_k^\dag y_N\|} \right)
\]

This inductive procedure yields a collection $[v_2], \dots, [v_n] \in L_q^n$,
and we let  $v_1 \in S^{2n-1}$   be such that $\linspan_\C \{v_1\} = \linspan_\C \{v_2, \dots, v_n\}^\bot$.
Finally  \[ \LPCA(Y) := \{ [v_1], \dots, [v_n] \} \] are the \textbf{Lens Principal Components} of $Y$.
Let $V_k \in \C^{n\times k}$ be the $n$-by-$k$ matrix with columns $v_1,\ldots, v_k$,
and let $P_k(Y) \subset L^k_q$ be the set of classes $\left[ \frac{V_k^\dag y_j}{\|V_k^\dag y_j\|}\right]$, $1 \leq j \leq  N$.
The point clouds $P_k(Y)$, $k=1,\ldots, n$, are the \textbf{Lens Principal Coordinates} of $Y$.

\subsection{Choosing a target dimension.}
The \textbf{variance recovered} by the first $k$ Lens Principal Components $[v_1], \dots, [v_k] \in L^n_q$
is  defined as 
\[  
\var_k(Y) := \frac{1}{N}\sum_{l=2}^k \sum_{j=1}^N d_L
\left(
\left[\frac{V_l^\dag y_j}{\|V_l^\dag y_j\|}\right], L_q^{l-1}(e_{l-1}) 
\right)^2 
\] 
where $V_l$ is the $n$-by-$l$ matrix with columns $v_1, \ldots, v_l$, $1<l\leq k$, 
and $e_{l-1} \in \C^{l}$ is the vector $[0,\ldots, 0, 1, 0]$. 

Therefore, the \textbf{percentage of cumulative variance} $p.\var(k) := \var_k (Y) \big/\var_n (Y)$, can be interpreted as the portion 
of total variance of $Y$ along  $\LPCA(Y)$, explained by the first $k$ components.

Thus we can select the target dimension as the smallest $k$ for which $p.\var_k(Y)$ is greater than a predetermined value. In other words, we select the dimension that recovers a significant portion of the total variance.
Another possible guideline to choose the target dimension is as the minimum value of $k$ for which $p.\var(k) - p.\var(k+1) < \gamma$ for a small $\gamma>0$.

\subsection{Independence of the cocycle representative.}

Let  $\eta \in Z^1(\mathcal{N}(\mathcal{B}_\epsilon); \Z_q)$ be such that $[\eta] \neq 0$ in  $H^1(\mathcal{N}(\mathcal{B}_\epsilon); \Z_q)$, 
and let $\eta' = \eta + \delta^0(\alpha)$ with $\alpha \in C^0(\mathcal{N}(\mathcal{B}_\epsilon); \Z_q)$. If $b\in U_j$, then \begin{align*}
    f_{\eta'}(b) 
    & =  [ \sqrt{\phi_1(b)}\zeta_q^{\eta_{j1} + \alpha_1} : \cdots : \sqrt{\phi_n(b)}\zeta_q^{\eta_{jn} + \alpha_n } ]
\end{align*}

If $Z_\alpha$
is  the square diagonal matrix with entries
$\zeta_q^{\alpha_1} , \zeta_q^{\alpha_2},\ldots,\zeta_q^{\alpha_n}  $, then  $ f_{\eta'}(b) = Z_\alpha \cdot f(b) $. Moreover, 
after taking classes in $L^n_q$, this implies that $f_{\eta'}(X) = Z_\alpha \cdot f(X)$.
Since
$\cov (Z_\alpha \cdot f(X) )=  Z_\alpha\cov(f(X)) Z_\alpha^\dag$ and   $Z_\alpha$ is orthonormal, then if $v$ is an eigenvector of $\cov(f(X) )$ with eigenvalue $\sigma$, we also have that $Z_\alpha v$ is an eigenvector of $\cov (Z_\alpha \cdot f(X) )$ with the same eigenvalue. Therefore \[ \LastLensComp( f_{\eta'}(X) )  = Z_\alpha \LastLensComp(f(X)). \]

Since each component in $\LPCA$ is obtained in the same manner, we have that $ \LPCA(f_{\eta'}(X)) = Z_\alpha \LPCA(f(X)). $ Thus,
 the lens coordinates from two  cohomologous cocycles $\eta $ and $\eta + \delta^0(\alpha)$ (i.e., representing the same  cohomology class) only differ by the  isometry of $L_q^n$ induced by the linear map $Z_\alpha$.

\subsection{Visualization map for $L_3^2$.} 

Given $v_1, \dots, v_n \in S^{2n-1}$ representatives for the classes in $\LPCA(Y)$. We want to visualize $P_2(Y) \subset L_3^2$ in the fundamental domain described in \Cref{subsec:fundamental_domain}. Let
\[ P_2(Y) =
\big\{ \big[
\< y_i, v_1\>_\C ,
\< y_i, v_2\>_\C \big] \in S^3\subset \C^2 : [y_i]\in Y
\big\} \] and define
$G : P_2(Y) \longrightarrow S^3 \subset \C^2$ to be \begin{equation}\label{eqn:visualization_map}
    G(z, w) := \left( \zeta_3^{-k} z, \left( \arg(w) - \frac{\pi}{3} \right) \sqrt{1 - \norm{z}^2} \right)
\end{equation} where $\arg(w)\in \left[ 0, \frac{2\pi}{3}\right)$, and $k$ an integer such that \[ \arg(z) = k\frac{2\pi}{3} + \theta,\] where $\theta$ is the remainder after division by $\frac{2\pi}{3}$.

\section{Examples}

\subsection{The Circle $S^1$}

Let $S^1 \subset \C$ be the unit circle, and let $X$ a random sample around $S^1$, with $10,000$ points and Gaussian noise in the normal direction. $L\subset X$ is a landmark set with 10 points obtained as described in \Cref{subsec:landmark}.

\begin{figure}[!htb]
    \centering
    \includegraphics[width=0.5\textwidth]{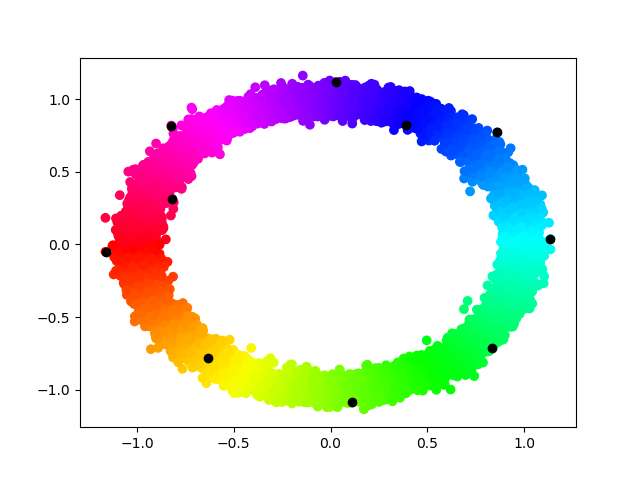}\includegraphics[width=0.5\textwidth]{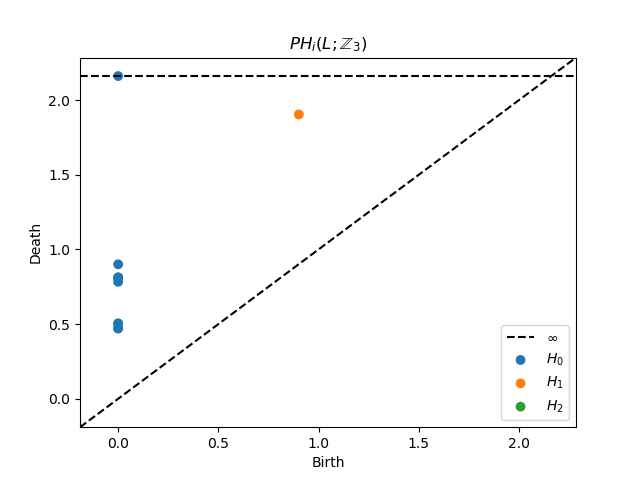}
    \caption{\textbf{Left:} Sample $X$, in black landmark set $L\subset X$. \textbf{Right:} $PH^i(\mathcal{R}(L); \Z_3)$ for $i=0,1,2$.}
    \label{fig:s1}
\end{figure}



Let $a$ be the cohomological death of the most persistent class $PH^1(\mathcal{R}(L); \Z_q)$. For $\epsilon := a + 10^{-5}$ and $\eta = i^*(\eta') \in Z^1(\mathcal{N}(\mathcal{B}_\epsilon); \Z_q)$ we define the map $f :B_\epsilon \rightarrow L_3^{10}$ as in \Cref{eqn:f}.

After computing $\LPCA$ for $f(X) \subset L_3^{10}$ and the percentage of cumulative variance $p.\var_Y(k)$ we obtain the row in \Cref{tab:cum_varaince} with label $S^1$ (see \Cref{fig:s1_cumulative_variance} for more details). We see that dimension $1$ recovers $\sim 60\%$ of the variance. Moreover, \Cref{fig:s1_embedding} shows $P_2(f(X)) \subset L_3^2$ in the fundamental domain described in \Cref{subsec:fundamental_domain} trough the map in \Cref{eqn:visualization_map}.

\begin{figure}[!htb]
    \centering
    \includegraphics[width=0.6\textwidth]{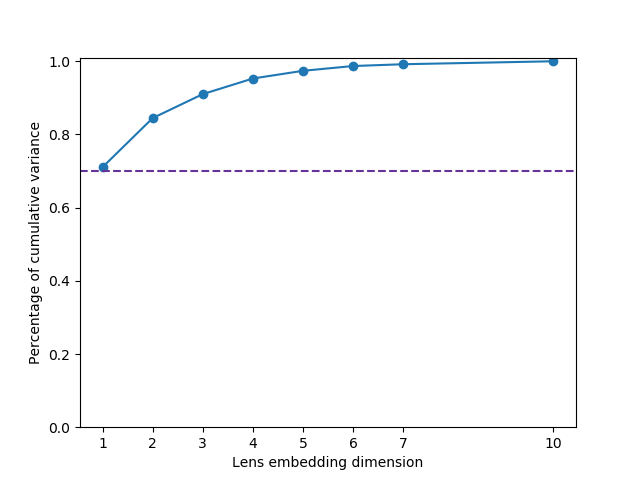}
    \caption{Profile of recovered variance on $S^1$.}
    \label{fig:s1_cumulative_variance}
\end{figure}

\begin{table}[!htb]
\centering
\begin{tabular}{|l|l|l|l|l|l|}
\hline
\textbf{Dim. ($n$)}                      & \textbf{1}  & \textbf{2}  & \textbf{3}  & \textbf{4}  & \textbf{5}    \\ \hline
\textbf{$S^1$}      & 0.62        & 0.75        & 0.81        & 0.86        & 0.89        \\ \hline
\textbf{$M(\Z_3,1)$} & 0.56        & 0.7         & 0.76        & 0.8         & 0.83        \\ \hline
\textbf{$L_3^2$}   & 0.47        & 0.62        & 0.67        & 0.71        & 0.73        \\ \hline
\end{tabular}
\caption{Percentage of recovered variance in $L^n_3$.}
\label{tab:cum_varaince}
\end{table}

\begin{figure}[!htb]
    \centering
    \includegraphics[width=0.9\textwidth]{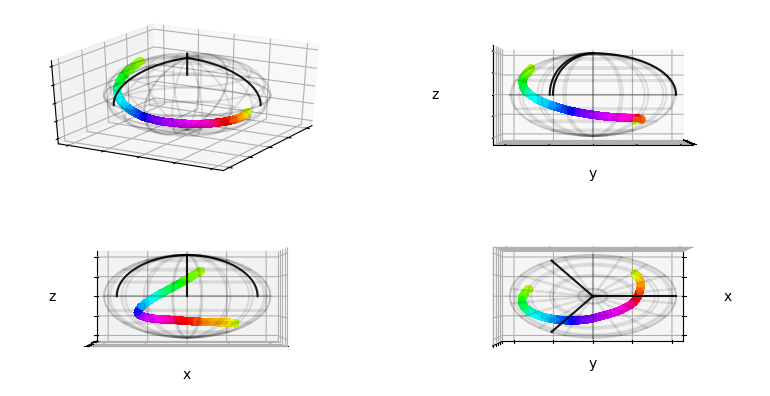}
    \caption{Visualization $P_2 (f(X)) \subset L_3^2$.}
    \label{fig:s1_embedding}
\end{figure}

One key aspect of LC (Lens coordinates) is that it is designed to highlight the cohomology class $\eta$ used on \Cref{eqn:f}. This is easily observed in this example; 
we selected the most persistent class in $PH^1(\mathcal{R}(L); \Z_3)$ and as a consequence in \Cref{fig:s1_embedding} we see how this class is preserved while all the information in the normal direction is lost in the process.


\subsection{The Moore space $M(\Z_3, 1)$.}

Let $G$ be an abelian group and $n\in \N$. The Moore space $M(G, n)$ is a CW-complex such that $H_n(M(G, n), \Z) = G$ and $\tilde{H}_i(M(G, n), \Z) = 0$ for all $i \neq n$. A well known construction for $M(\Z_3, 1)$ can be found in \cite{hatcher2002algebraic}. For $x,y\in  \C$ with $|x|, |y| \leq 1$, we let

\begin{equation} \label{eqn:metric_moore}
    d(x,y) = \begin{dcases}
    \sqrt{\vert \<x,y\>_\R \vert } & \text{ if } \norm{x}, \norm{w} < 1 \\
    \min_{\zeta \in \Z_3} \sqrt{\vert \<x,\zeta y\>_\R \vert} & \text{ if } \norm{x} = 1 \text{ or } \norm{w} = 1 \\
    \min_{\zeta \in \Z_3} \arccos(\vert \<x,\zeta y\>_\R \vert) & \text{ if } \norm{x} = 1 \text{ and } \norm{w} = 1
    \end{dcases}.
\end{equation} \Cref{eqn:metric_moore} defines a metric on $M(\Z_3, 1)$,

\begin{figure}[!htb]
    \centering
    \includegraphics[width=0.5\textwidth]{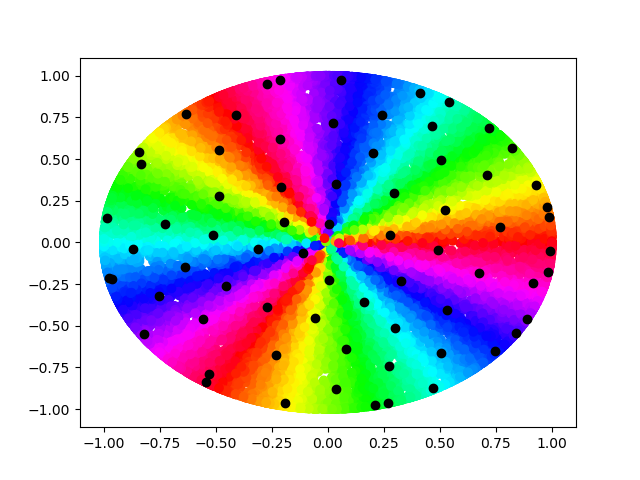}\includegraphics[width=0.5\textwidth]{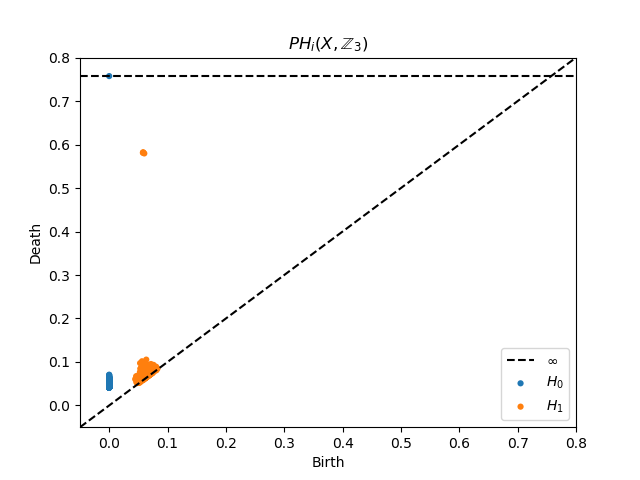}
    \caption{\textbf{Left:} $X\subset M(\Z_3, 1)$ with landmarks in black. \textbf{Right:} $PH^i(\mathcal{R}(L); \Z_3)$ for $i=0,1$.}
    \label{fig:moore}
\end{figure}

\Cref{fig:moore}, on the left, shows a sample $X\subset M(\Z_3, 1)$ with $\norm{X} = 15,000$ and $70$ landmarks. The landmarks were obtained by minmax sampling after feeding the algorithm with an initial set of $10$ point on the boundary on the disc.
\Cref{fig:moore_homology_2} shows the persistent cohomology of $\mathcal{R}(L)$ with coefficients in $\Z_2$ and $\Z_3$ side-by-side.

\begin{figure}[!htb]
    \centering
    \includegraphics[width=0.5\textwidth]{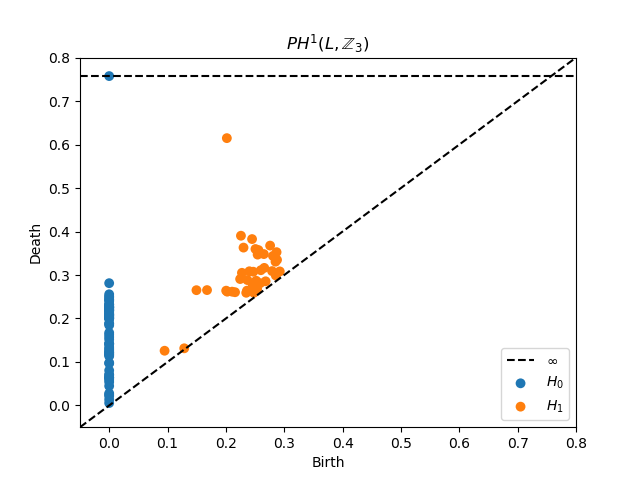}\includegraphics[width=0.5\textwidth]{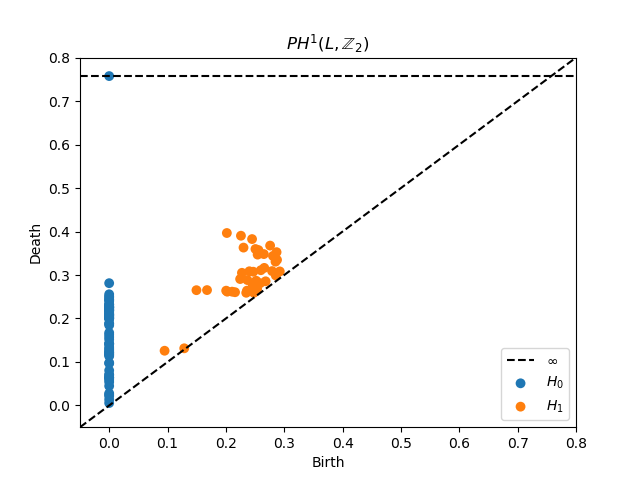}
    \caption{$PH^i(R(L); \F)$ for $i=0,1$ and $\F = \Z_2, \Z_3$.}
    \label{fig:moore_homology_2}
\end{figure}

We compute $ f: M(\Z_3, 1) \longrightarrow L_3^{70}$ analogously to the previous example and obtain a point cloud $f(X)\subset  L_3^{70}$. 
The profile of recovered variance is shown   in \Cref{tab:cum_varaince}. Dimension $2$ provides a low dimensional representation of $f(X)$ inside $L_3^2$ with $70\%$ of recovered variance (\Cref{fig:moore_cumulative_variance}).

\begin{figure}[!htb]
    \centering
    \includegraphics[width=0.6\textwidth]{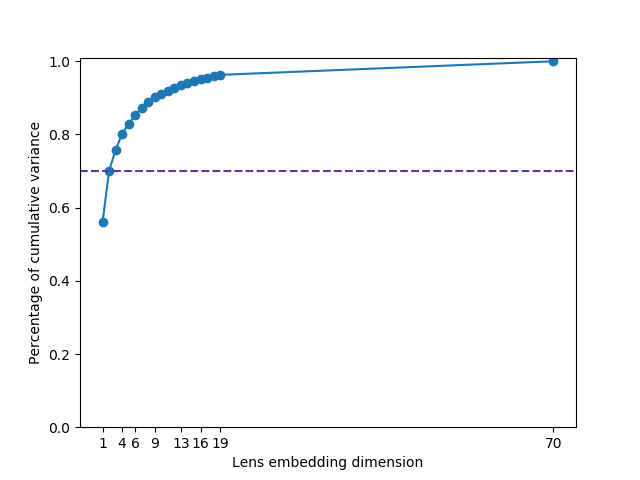}
    \caption{Profile of recovered variance on $M(\Z_3,1)$.}
    \label{fig:moore_cumulative_variance}
\end{figure}


\begin{figure}[!htb]
    \centering
    \includegraphics[width=0.9\textwidth, height=0.4\textwidth]{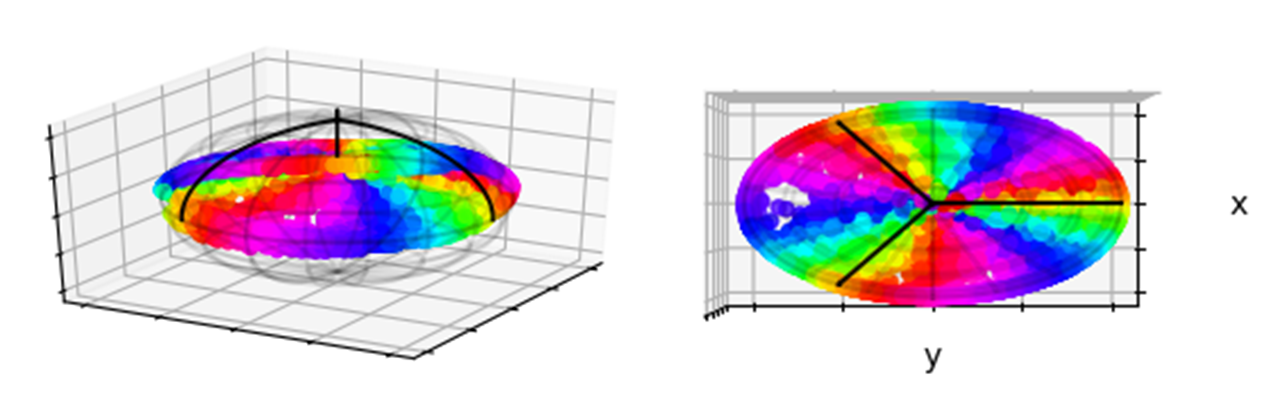}
    \caption{Visualization of the resulting $P_2 (f(X)) \subset L_3^2$.}
    \label{fig:moore_embedding}
\end{figure}

Since $f$ classifies the principal $\Z_3$-bundle $P_\eta$ over $M(\Z_3, 1)$, then $f$ must be homotopic to the inclusion of $M(\Z_q, 1)$ in $L_q ^\infty$. \Cref{fig:moore_embedding} shows $X\subset M(\Z_3,1)$ mapped by $f$ in $L_3^2$. Notice the identifications on $X$ are handled by the identification on $S^1\times\{0\} \subset D^3$ from the fundamental domain on \Cref{subsec:fundamental_domain}. See \url{https://youtu.be/_Ic730_xFkw} for a more complete visualization.

\subsection{The Lens space $L_3^2 = S^{3} / \Z_3$.}

We use the metric defined in \Cref{eqn:hausdorff_distance} on $L_3^2$ and randomly sample $15,000$ points to create $X\subset L_3^2$. \Cref{fig:lens}(left) shows the sample set using the fundamental domain from \cref{subsec:fundamental_domain}.

\begin{figure}[!htb]
    \centering
    \includegraphics[width=0.9\textwidth, height=0.4\textwidth]{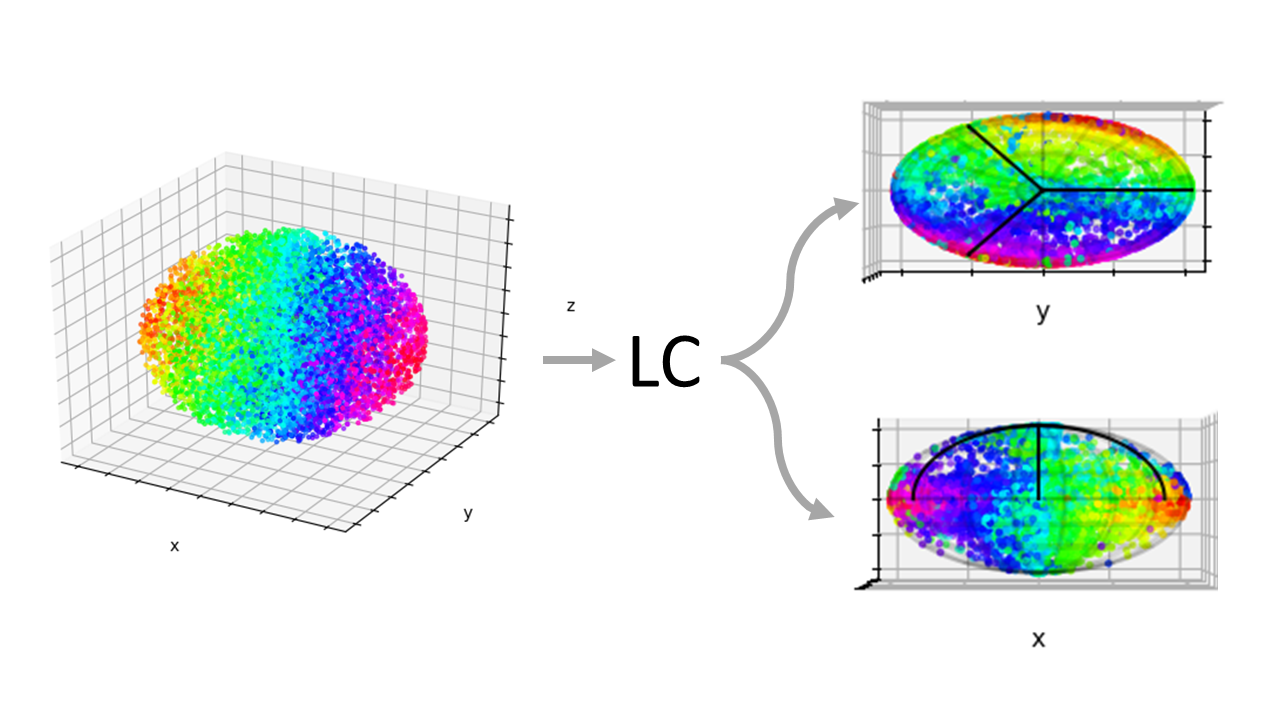}
    \caption{\textbf{Left:} $X\subset L_3^2$. \textbf{Right:} Lens coordinates.}
    \label{fig:lens}
\end{figure}

We can use $PH^i(\mathcal{R}(X); \Z_2)$ and $PH^i(\mathcal{R}(X); \Z_3)$ to verify that the sampled metric space has the expected topological features. \Cref{fig:lens_homology} contains the corresponding persistent diagrams.

\begin{figure}[!htb]
    \centering
    \includegraphics[width=0.5\textwidth]{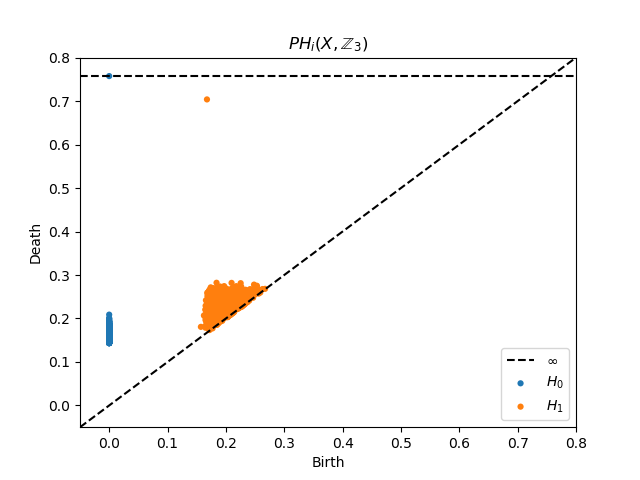}\includegraphics[width=0.5\textwidth]{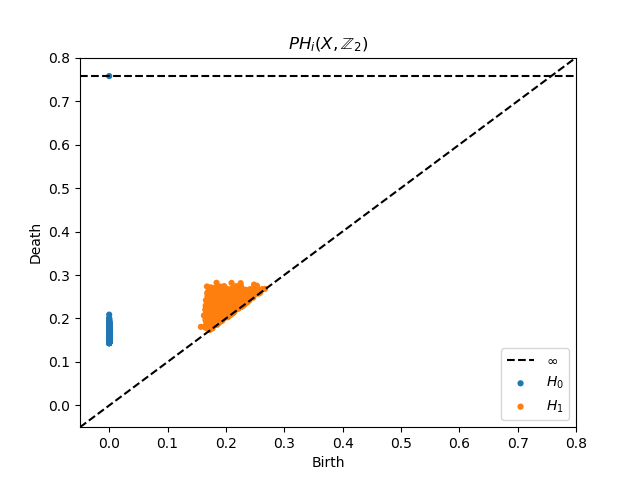}
    \caption{$PH^i(\mathcal{R}(L); \Z_3)$ for $i=0,1$. $PH^i(\mathcal{R}(L); \Z_2)$ for $i=0,1$.}
    \label{fig:lens_homology}
\end{figure}

Just as in the previous examples define $f: L_3^2 \rightarrow L_3^\infty$ using the most persistent class in $PH^1(\mathcal{R}(L); \Z_3)$. 
The homotopy class of $f$ must be the same as that of the inclusion $L_3^2 \subset  L_3^\infty$, 
since $f$ classifies the $\Z_3$-principal bundle $P_\eta$. Thus we expect $L_3^2$ to be preserved up to homotopy under $\LPCA$.
\Cref{fig:lens} offers a side and top view of $P_2 (f(X)) \subset L_3^2$. Here we clearly see how the original data set $X$ is transformed while preserving the identifications on the boundary of the fundamental domain.
Finally in \Cref{tab:cum_varaince} we show  the variance profile for the dimensionality reduction problem. 
We see that for dimension $4$ we have recovered more than $70\%$ of the total variance as seen in \Cref{tab:cum_varaince} and \Cref{fig:Lq_cumulative_variance}.

\begin{figure}[!htb]
    \centering
    \includegraphics[width=0.6\textwidth]{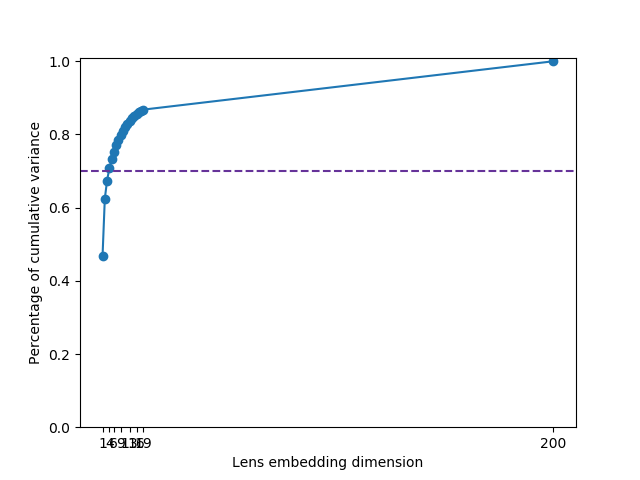}
    \caption{Profile of recovered variance on $L_3^2$.}
    \label{fig:Lq_cumulative_variance}
\end{figure}

\subsection{Isomap dimensionality reduction}

We conclude this section by providing evidence that Lens coordinates (LC) preserve topological features when compared to other dimensionality reduction algorithms. For this purpose we use Isomap (\cite{isomap}) as our point of comparison.

The Isomap algorithm consist of 3 main steps. The first step determines neighborhoods of each point using $k$-th nearest neighbors. The second step estimates the geodesic distances between all pairs of points using shortest distance path, and the final step applies classical MDS to the matrix of graph distances.

Let $\dgm$ be a persistent diagram. Define $\per_1$ to be the largest persistence of an element in $\dgm$, and let $\per_2$ be the second largest persistence of an element $\dgm$.

\begin{table}[!htb]
    \centering
    \begin{tabular}{|c|c|c|c|}
\hline
\multicolumn{2}{|c|}{$\per_1 / \per_2$}   & $\Z_2$                         & $\Z_3$                         \\ \hline
                               & Isomap & 1.0105                        & 1.0105                        \\ \cline{2-4}
\multirow{-2}{*}{$M(\Z_q, 1)$} & LC   & {\color[HTML]{000000} 1.7171} & {\color[HTML]{32CB00} 3.6789} \\ \hline
                               & Isomap & 1.0080                        & 1.0080                        \\ \cline{2-4}
\multirow{-2}{*}{$L_3^2$}      & LC   & {\color[HTML]{000000} 1.1592} & {\color[HTML]{32CB00} 2.8072} \\ \hline
\end{tabular}
    \caption{In green we highlight the fraction that indicates which method better identifies the topological features.}
    \label{tab:persistence_quotient}
\end{table}

For both $M(\Z_3,1)$ and $L_3^2$ it is clear that the Isomap projection fails to preserve the difference between the cohomology groups with coefficients in $\Z_2$ and $\Z_3$. On the other hand the LC projections maintains this difference in both examples (see \Cref{tab:moore_iso_v_lpca,,tab:lens_iso_v_lpca} for more details).

\begin{table*}[!htb]
    \centering
    \begin{tabular}{c|c|c}
         & Coefficients $\Z_2$ & Coefficients $\Z_3$ \\ \hline
        Isomap & \includegraphics[width=0.35\textwidth]{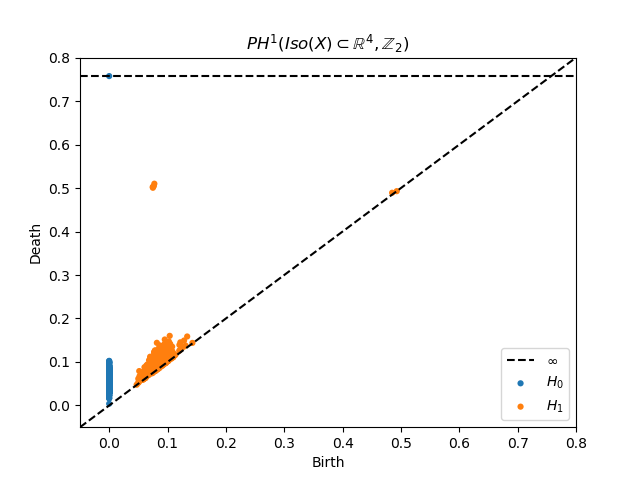} & \includegraphics[width=0.35\textwidth]{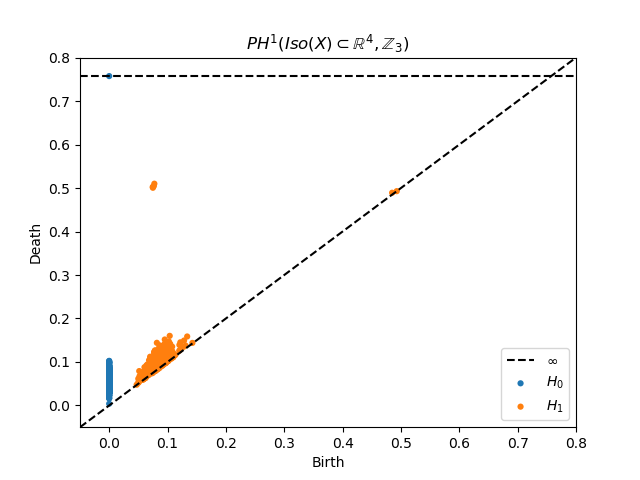} \\ \hline
        LC   & \includegraphics[width=0.35\textwidth]{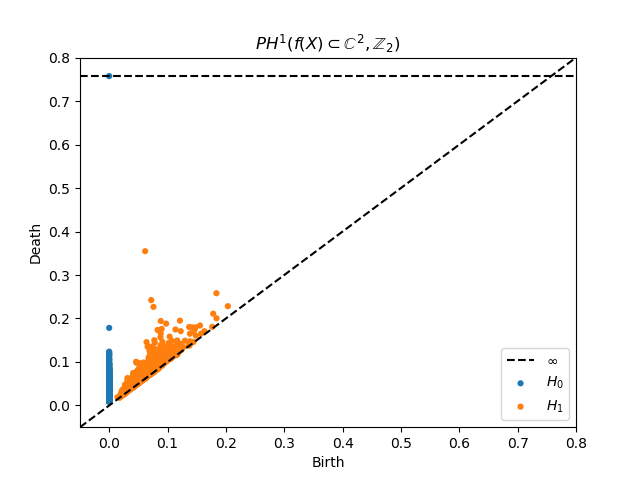} & \includegraphics[width=0.35\textwidth]{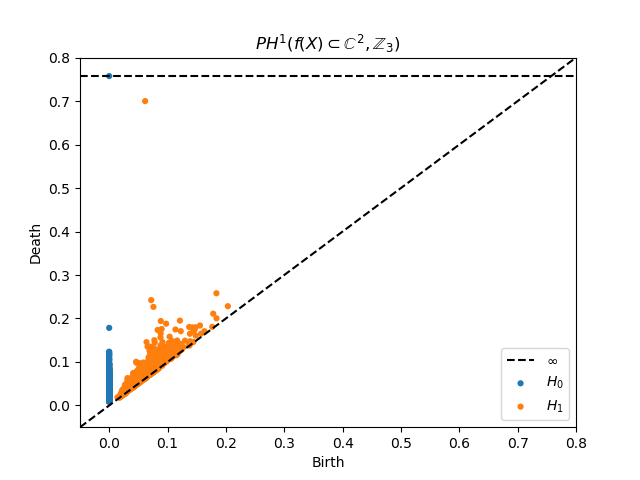}
    \end{tabular}
    \caption{Persistent homology of the Isomap vs. LPCA for $M(\Z_3, 1)$ into a $4$ dimensional space.}
    \label{tab:moore_iso_v_lpca}
\end{table*}

\begin{table*}[!htb]
    \centering
    \begin{tabular}{c|c|c}
         & Coefficients $\Z_2$ & Coefficients $\Z_3$ \\ \hline
        Isomap & \includegraphics[width=0.35\textwidth]{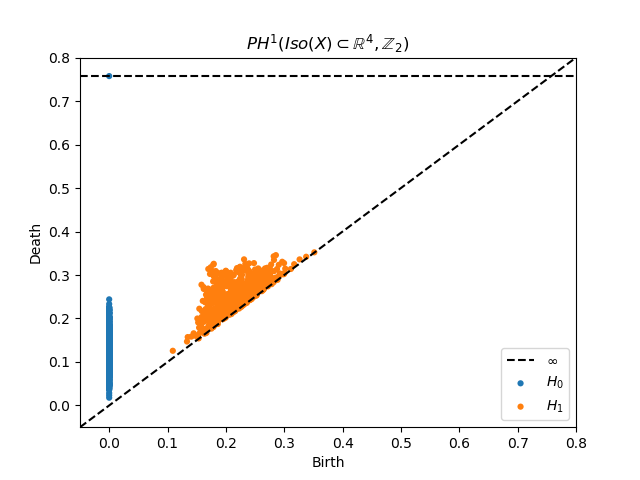} & \includegraphics[width=0.35\textwidth]{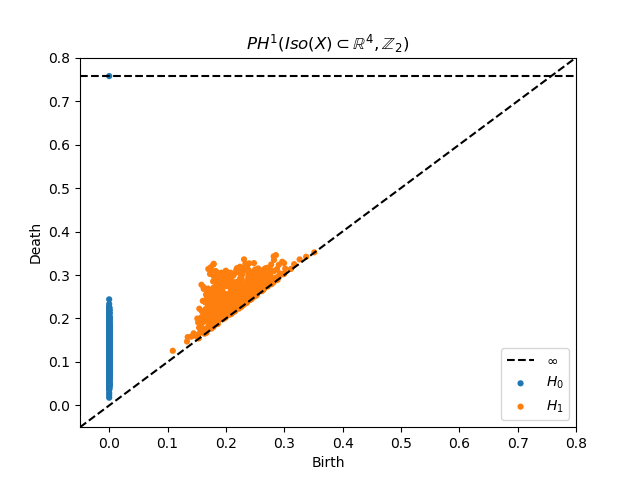} \\ \hline
        LC   & \includegraphics[width=0.35\textwidth]{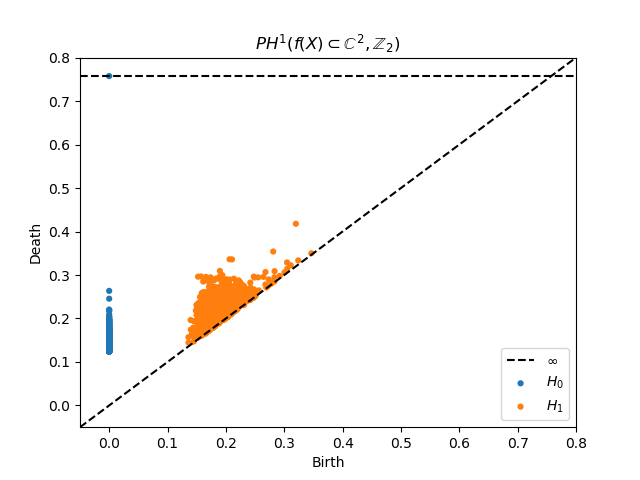} & \includegraphics[width=0.35\textwidth]{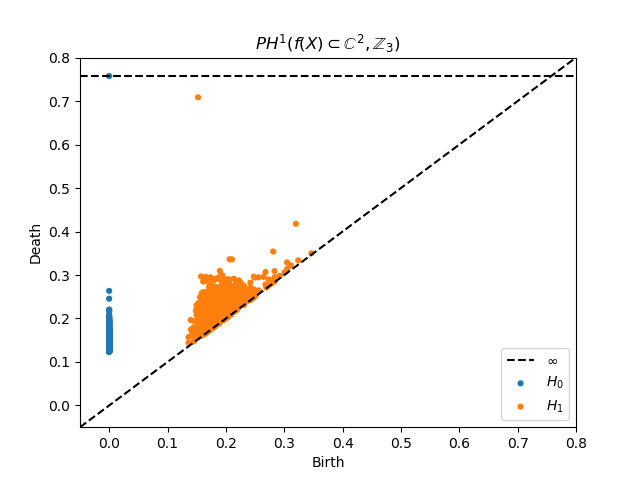}
    \end{tabular}
    \caption{Persistent homology of the Isomap vs. LPCA for $L_3^2$ into a $4$ dimensional space.}
    \label{tab:lens_iso_v_lpca}
\end{table*}






\subsection*{Acknowledgements}
This work was partially supported by the NSF under grant DMS-1622301.

\pagebreak

\bibliography{references}

\end{document}